\documentclass[11pt, twoside, leqno]{article}

\usepackage{amssymb}
\usepackage{amsmath}
\usepackage{amsthm}
\usepackage{color}
\usepackage{mathrsfs}

\usepackage{indentfirst}

\usepackage{txfonts}

\allowdisplaybreaks

\def\XXint#1#2#3{{\setbox0=\hbox{$#1{#2#3}{\int}$ }
\vcenter{\hbox{$#2#3$ }}\kern-.6\wd0}}

\def\({\left(}
\def \){ \right)}
\newtheorem{theorem}{Theorem}[section]
\newtheorem{lemma}[theorem]{Lemma}

\theoremstyle{definition}
\newtheorem{remark}[theorem]{Remark}

\renewcommand{\appendix}{\par
   \setcounter{section}{0}%
   \setcounter{subsection}{0}%
   \setcounter{subsubsection}{0}%
   \gdef\thesection{\@Alph\c@section}%
   \gdef\thesubsection{\@Alph\c@section.\@arabic\c@subsection}%
   \gdef\theHsection{\@Alph\c@section.}%
   \gdef\theHsubsection{\@Alph\c@section.\@arabic\c@subsection}%
   \csname appendixmore\endcsname
 }

\numberwithin{equation}{section}

\begin{document}

\arraycolsep=1pt

\title{\bf\Large Convergence of a Class of Schr\"{o}dinger Equations
\footnotetext{\hspace{-0.35cm} 2010 {\it
Mathematics Subject Classification}. Primary 35Q41.
\endgraf {\it Key words and phrases.} Schr\"{o}dinger equation, convergence, Sobolev spaces.
%\endgraf
%This work was partially supported by
 %NSFC-DFG  (Grant Nos.
%11761131002).
}
}
\author{ Dan Li and Haixia Yu\footnote{Corresponding author.}}
\date{}
\maketitle

\vspace{-0.7cm}

\begin{center}
\begin{minipage}{13cm}
{\small {\bf Abstract}\quad
%In this paper, we study the almost everywhere pointwise convergence of $a$-th order Schr\"{o}dinger equation and generalized Schr\"{o}dinger equation on $\mathbb{R}^{n+1}$ for all $a\in(0,\infty)$. It is a natural extension of %the work due to Sj\"{o}lin \cite{S2}.
In this paper, we set up the selection conditions for time series $\{t_k\}_{k=1}^\infty$ which converge to 0 as $k\rightarrow\infty$ such that the solutions of a class of generalized Schr\"odinger equations almost everywhere  pointwise converge to their initial data in $H^s(\mathbb{R}^n)$ for $s>0$. As it is known that the pointwise convergence can not be true for Schr\"odinger equation when $s<\frac{n}{2(n+1)}$ as $t\rightarrow0$.

}
\end{minipage}
\end{center}

%\tableofcontents

\section{Introduction}
We first consider the \emph{$a$-th order Schr\"{o}dinger equation} defined on $\mathbb{R}^{n+1}$ such that
\begin{equation}\label{1.1}
\left\{
\begin{aligned}
 iu_t+(-\Delta)^{\frac{a}{2}}u&=0,\quad &t>0,\\
u(x,0)&=f(x),\quad &t=0, \\
\end{aligned}
\right.
\end{equation}
where $a\in(0,\infty)$ and $f\in H^s(\mathbb{R}^n)$. The \emph{Sobolev spaces $H^s(\mathbb{R}^n)$} is defined as
$$H^s(\mathbb{R}^n):=\{f\in\mathcal{S}^\prime :\ \|f\|_{H^s(\mathbb{R}^n)}<\infty\},\quad s\in\mathbb{R}$$
with $$\|f\|_{H^s(\mathbb{R}^n)}:=\left(\int_{\mathbb{R}^n} (1+|\xi|^2)^s |\widehat{f}(\xi)|^2\,\textrm{d}\xi\right)^{\frac{1}{2}}.$$
The solution of \eqref{1.1} can formally be written as
$$u(x,t)=(2\pi)^{-n}S_{t,a}f(x).$$
Here
\begin{align}\label{eq:1.2}
S_{t,a}f(x):=\int_{\mathbb{R}^n} e^{ix\cdot\xi}e^{it|\xi|^a}\widehat{f}(\xi)\,\textrm{d}\xi, \quad x\in\mathbb{R}^n, \quad t\geq0,
\end{align}
and
$$\widehat{f}(\xi):=\int_{\mathbb{R}^n} e^{-i\xi\cdot x}f(x)\,\textrm{d}x, \quad \xi\in\mathbb{R}^n.$$

Carleson \cite{C} proposed a problem that determine the optimal $s_c$ such that
\begin{align}\label{eq:1.3}
\lim_{t\rightarrow0}(2\pi)^{-n}S_{t,2}f(x)=f(x)  \quad\textrm{a}.\textrm{e}.~ x\in\mathbb{R}^n
\end{align}
for all $f\in H^s(\mathbb{R}^n)$ with $s\geq s_c$. In the last four decades, Carleson Problem \eqref{eq:1.3} has attracted numerous attentions (see \cite{B3,B1,B2,DGLZ,DKWZ,L,MVV,S1,TV,V1} and the referees therein). And now it is known that $s_c=\frac{n}{2(n+1)}$ for $n\geq1$ is the critical index.  For $n=1$, Carleson \cite{C} set up \eqref{eq:1.3}  for $s\geq \frac{1}{4}$ and  Dahlberg and Kenig \cite{DK} showed that \eqref{eq:1.3} does not hold for $s<\frac{1}{4}$. For $n\geq 2$, Bourgain \cite{B2} formulated the counterexamples for all the case $s<\frac{n}{2(n+1)}$. Very recently, Du, Guth and Li \cite{DGL} set up \eqref{eq:1.3} for $s>\frac{1}{3}$ and $n=2$.  Du and Zhang \cite{DZ} then proved \eqref{eq:1.3} for $s>\frac{n}{2(n+1)}$ and $n\geq 2$.

There are some results for the $a$-th order Schr\"{o}dinger equation. If we replace $S_{t,2}$ in \eqref{eq:1.3} by $S_{t,a}$, we know that \eqref{eq:1.3} holds for $f\in H^{\frac{1}{2}}(\mathbb{R}^2)$ in the case that $n=2$ and $a>1$. When $n\geq 3$ and $a>1$, \eqref{eq:1.3} has been proved for $f\in H^s(\mathbb{R}^n)$ with $s>\frac{1}{2}$; see, for example, \cite{S1,V2}. Miao, Yang and Zheng \cite{MYZ} showed that \eqref{eq:1.3} holds for $s>\frac{3}{8}$ when $n=2$ and $a>1$. Recently, Cho and Ko \cite{CK} improved this result to $s>\frac{n}{2(n+1)}$ when $n\geq1$ and $a>1$.

On the other  hand, it is easy to see that $\lim_{t\rightarrow0}\|(2\pi)^{-n}S_{t,a}f(x)-f(x)\|_{L^2(\mathbb{R}^n)}=0$. From Riesz theorem there exists $\{t_k\}_{k=1}^\infty\rightarrow0$ as $k\rightarrow\infty$ such that
\begin{align}\label{eq:1.55}
\lim_{k\rightarrow\infty}(2\pi)^{-n}S_{t_k,a}f(x)=f(x),\quad\textrm{a}.\textrm{e}.~ x\in\mathbb{R}^n\end{align}
for $s\geq0$. In \cite{S2} Sj\"{o}lin showed that \eqref{eq:1.55} holds whenever
 $\sum_{k=1}^\infty {t_k}^{\frac{2s}{a}}<\infty$.
%Now, we want to treat this problem \eqref{eq:1.3} by a new angle, which originate from Sj\"{o}lin \cite{S2}. It is easy to see that $(2\pi)^{-n}S_{t,a}f(x)\rightarrow f(x)$ in $L^2(\mathbb{R}^n)$ as $t\rightarrow 0$ for all $a\in(0,\infty)$ and $f\in L^2(\mathbb{R}^n)$, which further implies %that there exists a sequence $\{t_k\}_{k=1}^\infty$ satisfying
%\begin{align}\label{eq:1.4}
%1>t_1>t_2>t_3>\cdot\cdot\cdot>0, \quad \lim_{k\rightarrow\infty} t_k=0
%\end{align}
%such that
%\begin{align}\label{eq:1.5}
%\lim_{k\rightarrow\infty}(2\pi)^{-n}S_{t_k,a}f(x)=f(x) \quad\textrm{a}.\textrm{e}.~ x\in\mathbb{R}^n
%\end{align}
%for all $a\in(0,\infty)$ and $f\in L^2(\mathbb{R}^n)$. It is natural to want to ask that whether \eqref{eq:1.5} holds for all $f\in L^2(\mathbb{R}^n)$ and $\{t_k\}_{k=1}^\infty$ satisfying \eqref{eq:1.4}. Indeed, we have known that this is impossible even for $f\in H^s(\mathbb{R}^n)$ unless $s>\frac{n}{2(n+1)}$. Therefore, we consider the problem of deciding for which sequence $\{t_k\}_{k=1}^\infty$ satisfying \eqref{eq:1.4} one has \eqref{eq:1.5} holds for all $f\in H^s(\mathbb{R}^n)$ with $s>0$. Sj\"{o}lin \cite{S2} proved \eqref{eq:1.5} holds for all $a>1$ and $f\in H^s(\mathbb{R}^n)$ with $s>0$ and $n\geq1$ under the condition that $\{t_k\}_{k=1}^\infty$ satisfying \eqref{eq:1.4} and $\sum_{k=1}^\infty {t_k}^{\frac{2s}{a}}<\infty$.
In this paper we try to consider this kind of pointwise convergence for some generalized Sch\"odinger equations. We first want to extend Sj\"olin's results to $0< a\leq1$.

\begin{theorem}\label{theorem 1.1}
Let $n\geq1$, $0<a<1,s\geq a$. Assume that $\sum_{k=1}^{\infty}{t_k}^2<\infty$, $f\in H^s(\mathbb{R}^n)$. Then
$$\lim_{k\rightarrow\infty}(2\pi)^{-n}S_{t_k,a}f(x)=f(x)$$
holds for almost everywhere $x\in\mathbb{R}^n$.
\end{theorem}

\begin{theorem}\label{theorem 1.2}
Let $n\geq1$, $0<s\leq a\leq 1$. Assume that $\sum_{k=1}^{\infty}{t_k}^{\frac{2s}{a}}<\infty$, $f\in H^s(\mathbb{R}^n)$. Then
$$\lim_{k\rightarrow\infty}(2\pi)^{-n}S_{t_k,a}f(x)=f(x)$$
holds for almost everywhere $x\in\mathbb{R}^n$.
\end{theorem}
A natural generalization of the pointwise convergence problem is to ask almost everywhere convergence along a wider approach region instead of vertical lines.
One of such problems may be non-tangential convergence to the initial data for $\textrm{a}.\textrm{e}.~ x\in\mathbb{R}^n$. That is, for $b>0$ and $f\in H^s(\mathbb{R}^n)$, for which $s$ such that
\begin{align}\label{eq:1.555}
\lim_{(y,t)\in\Gamma_b(x),(y,t)\rightarrow(x,0)}(2\pi)^{-n}S_{t,2}f(y)=f(x)  \quad\textrm{a}.\textrm{e}.~ x\in\mathbb{R}^n,
\end{align}
where $\Gamma_b(x)=\{(y,t)\in\mathbb{R}^{n+1}_{+}:|y-x|<b t\}.$
Sj\"{o}gren and Sj\"{o}lin \cite{SS} proved that \eqref{eq:1.555} fails for $s\leq\frac{n}{2}$. In fact, Sj\"{o}gren and Sj\"{o}lin \cite{SS} proved that there exists an $f\in H^{\frac{n}{2}}(\mathbb{R}^n)$ and a strictly increasing function $\Gamma$ with $\Gamma(0)=0$ such that for all $x\in\mathbb{R}^n$,
$$\lim_{|x-y|<\Gamma(t),t>0,(y,t)\rightarrow(x,0)}(2\pi)^{-n}|S_{t,2}f(y)|=\infty. $$
Another problem is to consider the relation between the degree of the tangency and regularity when $(x,t)$ approaches to $(x,0)$ tangentially.
One of the model problems raised by  Cho, Lee and Vargas \cite{CLV} is
\begin{align}\label{eq:1.77}
\lim_{t\rightarrow 0} S_{t,2}f(\Gamma(x,t))=f(x),  \quad\textrm{a}.\textrm{e}.~ x\in\mathbb{R}.
\end{align}
When $n=1$, here the curve $\Gamma(x,t)$ approaches $(x,0)$ tangentially to the hyperplane $\{(x,t):\ t=0\}$. Cho, Lee and Vargas \cite{CLV} set up \eqref{eq:1.77} for $s>\max\{\frac{1}{2}-\alpha,\frac{1}{4}\}$, here $\Gamma(x,t)$ satisfies H\"{o}lder condition of  order $\alpha$ in $t$ and Bilipschitz condition in $x$ with $0<\alpha\leq1$. $\alpha$ is essentially the degree of tangential convergence.
Ding and Niu \cite{DN} improved the result of Cho, Lee and Vargas \cite{CLV} to $s\geq\frac{1}{4}$ when $n=1$ and $\frac{1}{2}\leq\alpha\leq1$, but the problem is still open for $n\geq 2$. Recently, Li and Wang \cite{LW} proved convergence for $s>\frac{3}{8}$ when $n=2$ and $\Gamma(x,t):=x-\sqrt{t}\mu$, where $\mu$ is a unit vector in $\mathbb{R}^2$.
%By Cho, Lee and Vargas \cite{CLV}, a generalization of the pointwise convergence problem is to ask almost everywhere convergence along a wider approach region instead of vertical lines. One of the model problems raised by Cho, Lee and Vargas \cite{CLV} is
%\begin{align}\label{1.6}
%\lim_{t\rightarrow 0} S_{t,2}f(\Gamma(x,t))=f(x),  \quad\textrm{a}.\textrm{e}.~ x\in\mathbb{R}^n
%\end{align}
%for some continue curve with $\Gamma(x,0)=x$.
%One of such problem is to consider non-tangential convergence to the initial data, it was proved by Sj\"{o}gren and Sj\"{o}lin \cite{SS} that non-tangential convergence does not hold for $s\leq \frac{n}{2}$.
Next we consider the general case and define
\begin{align}\label{1.7}
S_{t,a}f(x+t^\beta\mu):=\int_{\mathbb{R}^n} e^{i(x\cdot\xi+t^\beta\mu\cdot\xi)}e^{it|\xi|^a}\widehat{f}(\xi)\,\textrm{d}\xi,  \quad t\geq0,  \quad a>0, \quad \beta\in\mathbb{R},
\end{align}
where $\mu$ is a unit vector in $\mathbb{R}^n$. For $S_{t,a}f(x+t^\beta\mu)$, we have that

\begin{theorem}\label{theorem 1.3}
Let $n\geq1$, $0<a<1$, $0<s\leq1$, $s>1-a$ when $\beta>1$ and $s>1-a\beta$ when $\beta\leq1$. Assume that $\sum_{k=1}^{\infty}{t_k}^{2(1+\frac{s-1}{a})}<\infty$ when $\beta>1$, $\sum_{k=1}^{\infty}{t_k}^{2(\beta+\frac{s-1}{a})}<\infty$ when $\beta\leq1$, $f\in H^s(\mathbb{R}^n)$. Then
$$\lim_{k\rightarrow\infty}(2\pi)^{-n}S_{t_k,a}f(x+t^\beta\mu)=f(x)$$
holds for almost everywhere $x\in\mathbb{R}^n$.
\end{theorem}

\begin{theorem}\label{theorem 1.4}
Let $n\geq1$, $a\geq1$, $0<s\leq a$, $s>a(1-\beta)$ when $\beta\leq1$. Assume that $\sum_{k=1}^{\infty}{t_k}^{\frac{2s}{a}}<\infty$ when $\beta>1$, $\sum_{k=1}^{\infty}{t_k}^{2(\beta-1+\frac{s}{a})}<\infty$ when $\beta\leq1$, $f\in H^s(\mathbb{R}^n)$. Then
$$\lim_{k\rightarrow\infty}(2\pi)^{-n}S_{t_k,a}f(x+t^\beta\mu)=f(x)$$
holds for almost everywhere $x\in\mathbb{R}^n$.
\end{theorem}

The last part of this paper is to discuss the convergence of a class of dispersive equations
\begin{equation}\label{1.8}
\left\{
\begin{aligned}
 iu_t+\gamma(\sqrt{-\Delta}) u&=0,\quad &t>0,\\
u(x,0)&=f(x),\quad &t=0, \\
\end{aligned}
\right.
\end{equation}
where $\gamma:\ \mathbb{R}^+\rightarrow \mathbb{R}^+$ is smooth, $f\in H^s(\mathbb{R}^n)$ with $n\geq1$.
%Many equations reduce to this type, for example, the wave equation $\gamma(|\xi|):=|\xi|$, the Schr\"{o}dinger equation $\gamma(|\xi|):=|\xi|^2$, the 4-order Schr\"{o}dinger equation  $\gamma(|\xi|):=|\xi|^2+|\xi|^4$, the Klein-Gordon equation $\gamma(|\xi|):=\sqrt{1+|\xi|^2}$, the Beam equation $\gamma(|\xi|):=\sqrt{1+|\xi|^4}$, the Boussinesq equation $\gamma(|\xi|):=|\xi|\sqrt{1+|\xi|^2}$ and the improved modified Boussinesq equation $\gamma(|\xi|):=\frac{|\xi|}{\sqrt{1+|\xi|^2}}$.
The solution of \eqref{1.8} can formally be written as
\begin{align}\label{1.9}
S_{t,\gamma}f(x):=\int_{\mathbb{R}^n} e^{ix\cdot\xi}e^{it\gamma(|\xi|)}\widehat{f}(\xi)\,\textrm{d}\xi,\quad t\geq0.
\end{align}
Next we consider the following pointwise convergence
\begin{align}\label{1.10}
\lim_{k\rightarrow\infty}(2\pi)^{-n}S_{t_k,\gamma}f(x)=f(x) \quad\textrm{a}.\textrm{e}.~ x\in\mathbb{R}^n.
\end{align}
Very recently, under some assumptions on $\gamma(|\xi|)$ Cho and Ko \cite{CK} set up \eqref{1.10} whenever $s>\frac{n}{2(n+1)}$. For example, $\gamma(|\xi|)=|\xi|^a$ with $a>1$.
%In a paper of Cho and Ko \cite{CK}, it was shown that \eqref{1.10} is true for each $s>\frac{n}{2(n+1)}$ and $n\geq 1$ with $\gamma$ satisfies some conditions.
%\begin{enumerate}
% \item[\rm(i)] $\gamma(|\xi|)$ is smooth on $\{\xi: |\xi|\approx 1\}$,
%  \item[\rm(ii)] $|D^\beta \gamma(|\xi|)|\lesssim|\xi|^{\alpha-|\beta|},\quad |\nabla \gamma(|\xi|)|\gtrsim |\xi|^{\alpha-1}$ for $\alpha>1$,
%  \item[\rm(iii)] $\det \partial_{\xi\xi}\gamma(|\xi|)$ is positive definite.
%\end{enumerate}
These assumptions also appeared in \cite{B1} and \cite{L}.  For $S_{t,\gamma}$, we have that

\begin{theorem}\label{theorem 1.5}
Let $n\geq1$, $0<s\leq1$, $\gamma(t)\geq0$ and $\frac{\gamma(t)}{t}$  be smooth and increasing on $(0,+\infty)$.
%$\gamma(t):\ \mathbb{R}^+\rightarrow \mathbb{R}^+$ and $\frac{\gamma(t)}{t}:\ \mathbb{R}^+\rightarrow \mathbb{R}^+$ be smooth and increasing on $(0,\infty)$ with $\gamma(0)\geq0$.
Assume that $\sum_{k=1}^{\infty}\frac{1}{[\gamma^{-1}(\frac{\gamma(1)}{t_k})]^{2s}}<\infty$, $f\in H^s(\mathbb{R}^n)$. Then
$$\lim_{k\rightarrow\infty}(2\pi)^{-n}S_{t_k,\gamma}f(x)=f(x)$$
holds for almost everywhere $x\in\mathbb{R}^n$, where $\gamma^{-1}$ is the inverse function of $\gamma$.
\end{theorem}

\begin{remark}\label{remark 1.6}
If we take  $\gamma(|\xi|):=|\xi|^a$ with $a>1$, we come back to Theorem 1 in Sj\"{o}lin \cite{S2}. Further more, if we take  $\gamma(|\xi|):=|\xi|$, Theorem \ref{theorem 1.5} does also hold. If we take $\gamma(|\xi|):=|\xi|\sqrt{1+|\xi|^2}$, we have the \emph{Boussinesq equation} which is defined on $\mathbb{R}^{n+1}$ such that
\begin{equation}\label{1.9}
\left\{
\begin{aligned}
 iu_t+\sqrt{-\Delta}\sqrt{1-\Delta} u&=0,\quad &t>0,\\
u(x,0)&=f(x),\quad &t=0, \\
\end{aligned}
\right.
\end{equation}
for all $f\in H^s(\mathbb{R}^n)$. Its solution can formally be written as
$$S_{t,\sqrt{-\Delta}\sqrt{1-\Delta}}f(x):=\int_{\mathbb{R}^n} e^{ix\cdot\xi}e^{it|\xi|\sqrt{1+|\xi|^2}}\widehat{f}(\xi)\,\textrm{d}\xi,\quad t\geq0.$$
From Theorem \ref{theorem 1.5}, we have the following conclusion.
Let $n\geq1$, $0<s\leq1$. Assume that $\sum_{k=1}^\infty {t_k}^s<\infty$,
%$\sum_{k=1}^{\infty}\frac{2^s}{[\sqrt{1+\frac{8}{ {t_k}^2}}-1]^{s}}<\infty$,
$f\in H^s(\mathbb{R}^n)$. Then
$$\lim_{k\rightarrow\infty}(2\pi)^{-n}S_{t_k,\sqrt{-\Delta}\sqrt{1-\Delta}}f(x)=f(x)$$
holds for almost everywhere $x\in\mathbb{R}^n$.
If we take $\gamma(|\xi|):=|\xi|^2+|\xi|^4$, we have the \emph{4-order Schr\"{o}dinger equation} which is defined on $\mathbb{R}^{n+1}$ such that
\begin{equation}\label{1.12}
\left\{
\begin{aligned}
 iu_t+[(-\Delta)+(-\Delta)^2] u&=0,\quad &t>0,\\
u(x,0)&=f(x),\quad &t=0, \\
\end{aligned}
\right.
\end{equation}
for all $f\in H^s(\mathbb{R}^n)$. Its  solution can formally be written as
$$S_{t,(-\Delta)+(-\Delta)^2} f(x):=\int_{\mathbb{R}^n} e^{ix\cdot\xi}e^{it(|\xi|^2+|\xi|^4)}\widehat{f}(\xi)\,\textrm{d}\xi,\quad t\geq0.$$
From Theorem \ref{theorem 1.5}, we have the following result.
Let $n\geq1$, $0<s\leq1$. Assume that $\sum_{k=1}^\infty {t_k}^{\frac{s}{2}}<\infty$,
%$\sum_{k=1}^{\infty}\frac{2^s}{[\sqrt{1+\frac{8}{ {t_k}}}-1]^{s}}<\infty$,
$f\in H^s(\mathbb{R}^n)$. Then
$$\lim_{k\rightarrow\infty}(2\pi)^{-n}S_{t,(-\Delta)+(-\Delta)^2}  f(x)=f(x)$$
holds for almost everywhere $x\in\mathbb{R}^n$.
%Theorem \ref{theorem 1.5} is a generalise of Theorem 1 in Sj\"{o}lin \cite{S2}, since the equation \eqref{1.8} reduces to the $a$-th order Schr\"{o}dinger equation and the condition that $\sum_{k=1}^{\infty}\frac{1}{[\gamma^{-1}(\frac{\gamma(1)}{t_k})]^{2s}}<\infty$ is equivalent to $\sum_{k=1}^\infty {t_k}^{\frac{2s}{a}}<\infty$ when $\gamma(|\xi|):=|\xi|^a$ with $a>1$.
\end{remark}

Finally, we want to extend the Theorem \ref{theorem 1.3} and Theorem \ref{theorem 1.4} to general case, so we define
\begin{align}\label{1.99}
S_{t,\gamma}f(x+t^\beta\mu):=\int_{\mathbb{R}^n} e^{i(x\cdot\xi+t^\beta\mu\cdot\xi)}e^{it\gamma(|\xi|)}\widehat{f}(\xi)\,\textrm{d}\xi,\quad t\geq0, \quad \beta\in\mathbb{R},
\end{align}
where $\mu$ is a unit vector in $\mathbb{R}^n$.
Next we consider the following pointwise convergence
\begin{align}\label{1.100}
\lim_{k\rightarrow\infty}(2\pi)^{-n}S_{t,\gamma}f(x+t^\beta\mu)=f(x) \quad\textrm{a}.\textrm{e}.~ x\in\mathbb{R}^n.
\end{align}
For $S_{t,\gamma}f(x+t^\beta\mu)$, we have that
\begin{theorem}\label{theorem 1.55}
Let $n\geq1$, $0<s\leq1$, $\gamma(t)$ be the same as in Theorem \ref{theorem 1.5}.
%$\gamma(t):\ \mathbb{R}^+\rightarrow \mathbb{R}^+$ and $\frac{\gamma(t)}{t}:\ \mathbb{R}^+\rightarrow \mathbb{R}^+$ be smooth and increasing on $(0,\infty)$ with $\gamma(0)\geq0$.
Assume that $\sum_{k=1}^{\infty}\frac{1}{[\gamma^{-1}(\frac{\gamma(1)}{t_k})]^{2s}}<\infty$ when $\beta>1$ and
$\sum_{k=1}^{\infty}\frac{{t_k}^{\beta-1}}{[\gamma^{-1}(\frac{\gamma(1)}{t_k})]^{2s}}<\infty$ when $\beta\leq1$,
$f\in H^s(\mathbb{R}^n)$. Then
$$\lim_{k\rightarrow\infty}(2\pi)^{-n}S_{t_k,\gamma}f(x+t^\beta\mu)=f(x)$$
holds for almost everywhere $x\in\mathbb{R}^n$, where $\gamma^{-1}$ is the inverse function of $\gamma$.
\end{theorem}

Throughout this paper, we use $C$ to denote a positive constant which is independent of the essential variables but its value may be different from line to line. $f\lesssim g$ means $f\leq Cg$ and $f\thickapprox g$ means $f\lesssim g\lesssim f$.
%Throughout this paper, we will write $C$ to denote a \emph{positive
%constant} that is independent of the essential variables involved, but whose
%value may be different on each line. For two real functions $f$ and $g$, we will always use $f\ls g$ or $g\gs f$ to mean that $f\le Cg$ and use $f\approx g$ to mean that $f\ls g\ls f$.

\section{Proof of the main results}
%To begin, we first prove Theorem \ref{theorem 1.1}.

\begin{proof}[Proof of Theorem \ref{theorem 1.1}]
Without loss of generality we can assume $0\leq t_k<1$.
Let $g\in\mathcal{S}(\mathbb{R}^n)$,
so $$(2\pi)^{-n}S_{t_k,a}g(x)=(2\pi)^{-n}\int_{\mathbb{R}^n} e^{ix\cdot\xi}e^{it_k|\xi|^a}\widehat{g}(\xi)\,\textrm{d}\xi.$$
Since $$g(x)=(2\pi)^{-n}\int_{\mathbb{R}^n} e^{ix\cdot\xi}\widehat{g}(\xi)\,\textrm{d}\xi.$$
Then $$(2\pi)^{-n}S_{t_k,a}g(x)-g(x)=(2\pi)^{-n}\int_{\mathbb{R}^n} e^{ix\cdot\xi}\left(e^{it_k|\xi|^a}-1\right)\widehat{g}(\xi)\,\textrm{d}\xi.$$
We define $\widehat{\varphi}(\xi):=\left(1+|\xi|^2\right)^{\frac{s}{2}}\widehat{g}(\xi)$.
Then
\begin{eqnarray*}
(2\pi)^{-n}S_{t_k,a}g(x)-g(x)
&=&(2\pi)^{-n}\int_{\mathbb{R}^n} e^{ix\cdot\xi}\left(e^{it_k|\xi|^a}-1\right)\left(1+|\xi|^2\right)^{-\frac{s}{2}}\widehat{\varphi}(\xi)\,\textrm{d}\xi\\
&=&(2\pi)^{-n}\int_{\mathbb{R}^n} e^{ix\cdot\xi}m_a(\xi)\widehat{\varphi}(\xi)\,\textrm{d}\xi,
\end{eqnarray*}
where $m_a(\xi)=\frac{e^{it_k|\xi|^a}-1}{(1+|\xi|^2)^{\frac{s}{2}}}$. Then
$$\left|m_a(\xi)\right|=\left|\frac{e^{it_k|\xi|^a}-1}{(1+|\xi|^2)^{\frac{s}{2}}}\right|\lesssim\frac{t_k|\xi|^a}{(1+|\xi|^2)^{\frac{s}{2}}}\leq t_k$$
for $a\leq s$.
We obtain $\|m_a\|_{L^\infty(\mathbb{R}^n)}\lesssim t_k$.
According to Plancherel's theorem we can get
$$\|(2\pi)^{-n}S_{t_k,a}g(x)-g(x)\|_{L^2(\mathbb{R}^n)}^2
=\|m_a\widehat{\varphi}\|_{L^2(\mathbb{R}^n)}^2\leq\|m_a\|_{L^\infty(\mathbb{R}^n)}^2\|\widehat{\varphi}\|_{L^2(\mathbb{R}^n)}^2\lesssim {t_k}^2\|g\|_{H^s(\mathbb{R}^n)}^2.$$
So far we have proved that
$$\|(2\pi)^{-n}S_{t_k,a}g(x)-g(x)\|_{L^2(\mathbb{R}^n)}\lesssim {t_k}\|g\|_{H^s(\mathbb{R}^n)},\quad g\in\mathcal{S}(\mathbb{R}^n).$$
Next we want to prove
$$\|h_{k,a}\|_{L^2(\mathbb{R}^n)}
=\|(2\pi)^{-n}S_{t_k,a}f(x)-f(x)\|_{L^2(\mathbb{R}^n)}
\lesssim {t_k}\|f\|_{H^s(\mathbb{R}^n)},\quad f\in H^s(\mathbb{R}^n),$$
where $h_{k,a}(x)=(2\pi)^{-n}S_{t_k,a}f(x)-f(x)$.
Since the Schwartz space $\mathcal{S}(\mathbb{R}^n)$ is dense in $H^s(\mathbb{R}^n)$,
so for any function $f\in H^s(\mathbb{R}^n)$, we can find $g_k\in \mathcal{S}(\mathbb{R}^n)$ satisfying that
$$\|g_k-f\|_{H^s(\mathbb{R}^n)}\leq\frac{1}{100}t_k\|f\|_{H^s(\mathbb{R}^n)}.$$
This implies that
\begin{eqnarray*}
\|h_{k,a}\|_{L^2(\mathbb{R}^n)}
&=&\|(2\pi)^{-n}S_{t_k,a}f(x)-f(x)\|_{L^2(\mathbb{R}^n)}\\
&\leq&\|(2\pi)^{-n}S_{t_k,a}(f-g_k)(x)\|_{L^2(\mathbb{R}^n)}+\|(2\pi)^{-n}S_{t_k,a}(g_k)(x)-{g_k}(x)\|_{L^2(\mathbb{R}^n)}+\|g_k(x)-f(x)\|_{L^2(\mathbb{R}^n)}\\
&\lesssim&\|f-g_k\|_{L^2(\mathbb{R}^n)}+t_k\|g_k\|_{H^s(\mathbb{R}^n)}+\|f-g_k\|_{L^2(\mathbb{R}^n)}\\
&\lesssim&\|f-g_k\|_{H^s(\mathbb{R}^n)}+t_k\left(\|f-g_k\|_{H^s(\mathbb{R}^n)}+\|f\|_{H^s(\mathbb{R}^n)}\right)\\
&\leq&\frac{1}{100}t_k\|f\|_{H^s(\mathbb{R}^n)}+\frac{1}{100}{t_k}^2\|f\|_{H^s(\mathbb{R}^n)}+t_k\|f\|_{H^s(\mathbb{R}^n)}\\
&\lesssim&t_k\|f\|_{H^s(\mathbb{R}^n)}.
\end{eqnarray*}
It follows that
$$\sum_{k=1}^\infty\int_{\mathbb{R}^n}|h_{k,a}|^2\,\textrm{d}x\lesssim\sum_{k=1}^\infty {t_k}^2\|f\|_{H^s(\mathbb{R}^n)}^2<\infty.$$
 By Fubini's theorem we have $\int_{\mathbb{R}^n}\sum_{k=1}^\infty|h_{k,a}|^2\,\textrm{d}x<\infty$, which further implies that $\sum_{k=1}^\infty|h_{k,a}|^2<\infty$, a.e. $x\in\mathbb {R}^n$. Hence $\lim_{k\rightarrow\infty}h_{k,a}(x)=0$, a.e. $x\in\mathbb {R}^n$.
Therefore,
$$\lim_{k\rightarrow\infty}(2\pi)^{-n}S_{t_k,a}f(x)=f(x)$$
for almost everywhere $x\in\mathbb{R}^n$. This completes the proof of Theorem \ref{theorem 1.1}.
\end{proof}
The proof of Theorem \ref{theorem 1.2} relies on the following lemma and Sj\"{o}lin \cite{S2} considered the case of $a>1$. Here we assume $0<s\leq a\leq1$.
%We now turn to the proof of Theorem \ref{theorem 1.2} with Lemma \ref{lemma 2.2}.

\begin{lemma}\label{lemma 2.2}
Let $n\geq1$, $0<s\leq a\leq1$, $0<\delta<1$ and $m_a(\xi)=\frac{e^{i\delta|\xi|^a}-1}{(1+|\xi|^2)^{\frac{s}{2}}}$. Then $\|m_a\|_{L^\infty(\mathbb{R}^n)}\leq C\delta^{\frac{s}{a}}$, where the constant $C$ does not depend on $\delta$.
\end{lemma}
\begin{proof}
The proof can be divided into three situations.
\begin{enumerate}
  \item[\rm(i)] if $0\leq|\xi|\leq1$, we have
  $$|m_a(\xi)|=\left|\frac{e^{i\delta|\xi|^a}-1}{(1+|\xi|^2)^{\frac{s}{2}}}\right|\leq\delta|\xi|^a\leq\delta\leq\delta^{\frac{s}{a}}.$$
  \item[\rm(ii)] if $1<|\xi|<\delta^{-\frac{1}{a}}$, we obtain
  $$|m_a(\xi)|\leq\frac{\delta|\xi|^a}{|\xi|^s}=\delta|\xi|^{a-s}\leq\delta\delta^{-\frac{a-s}{a} }=\delta^{\frac{s}{a}}.$$
  \item[\rm(iii)] if $|\xi|\geq\delta^{-\frac{1}{a}}$, we get
  $$|m_a(\xi)|\lesssim\frac{1}{|\xi|^s}\leq\frac{1}{\delta^{-\frac{s}{a}}}=\delta^{\frac{s}{a}}.$$
\end{enumerate}
Altogether we have shown that $\|m_a\|_{L^\infty(\mathbb{R}^n)}\lesssim\delta^{\frac{s}{a}}$, which completes the proof of Lemma \ref{lemma 2.2}.
\end{proof}

\begin{proof}[Proof of Theorem \ref{theorem 1.2}]
Similarly to the proof of Theorem \ref{theorem 1.1}, we have that
$$h_{k,a}(x)=(2\pi)^{-n}\int_{\mathbb{R}^n} e^{ix\cdot\xi}m_a(\xi)\widehat{\varphi_1}(\xi)\,\textrm{d}\xi.$$
From Lemma \ref{lemma 2.2} we obtain $\|m_a\|_{L^\infty(\mathbb{R}^n)}\lesssim {t_k}^{\frac{s}{a}}$. We may conclude that
$$\|h_{k,a}\|_{L^2(\mathbb{R}^n)}^2\lesssim {t_k}^{\frac{2s}{a}}\|f\|_{H^s(\mathbb{R}^n)}^2.$$
This, combined with the fact that $\sum_{k=1}^{\infty}{t_k}^{\frac{2s}{a}}<\infty$, we assert that $\sum_{k=1}^\infty\int_{\mathbb{R}^n}|h_{k,a}|^2\,\textrm{d}x<\infty$. Applying Fubini's theorem again, we have $\int_{\mathbb{R}^n}\sum_{k=1}^\infty|h_{k,a}|^2\,\textrm{d}x<\infty$, which further implies that $\lim_{k\rightarrow\infty}h_{k,a}(x)=0$, a.e. $x\in\mathbb {R}^n$. Therefore,
$$\lim_{k\rightarrow\infty}(2\pi)^{-n}S_{t_k,a}f(x)=f(x)$$
for almost everywhere $x\in\mathbb{R}^n$. This completes the proof of Theorem \ref{theorem 1.2}.
\end{proof}

\begin{proof}[Proof of Theorem \ref{theorem 1.3}]
Similarly to the proof of Theorem \ref{theorem 1.1} and Theorem \ref{theorem 1.2}, it suffices to show that $\|m_{a,\mu}\|_{L^\infty(\mathbb{R}^n)}\lesssim\delta^{1+\frac{s-1}{a}}$ when $\beta>1$ and  $\|m_{a,\mu}\|_{L^\infty(\mathbb{R}^n)}\lesssim\delta^{\beta+\frac{s-1}{a}}$ when $\beta\leq1$
%with the bound does not depend on $\delta$, where $n\geq1$, $0<a<1$, $0<s\leq1$, $0<\delta<1$ and
for $m_{a,\mu}(\xi):=\frac{e^{i(\delta^\beta\mu\cdot\xi +\delta|\xi|^a)}-1}{(1+|\xi|^2)^{\frac{s}{2}}}$.

In fact,
when $\beta>1$,
\begin{enumerate}
  \item[\rm(i)] if $0\leq|\xi|\leq1$, we have
  $$|m_{a,\mu}(\xi)|=\left|\frac{e^{i(\delta^\beta\mu\cdot\xi +\delta|\xi|^a)}-1}{(1+|\xi|^2)^{\frac{s}{2}}}\right|\leq|\delta^\beta\mu\cdot\xi+\delta|\xi|^a|\leq\delta(|\xi|+|\xi|^a)\leq2\delta|\xi|^a\leq2\delta\leq2\delta^{1+\frac{s-1}{a}}.$$
  \item[\rm(ii)] if $1<|\xi|<\delta^{-\frac{1}{a}}$, we obtain
  $$|m_{a,\mu}(\xi)|\leq\frac{|\delta^\beta\mu\cdot\xi+\delta|\xi|^a|}{|\xi|^s}\leq\frac{\delta(|\xi|+|\xi|^a)}{|\xi|^s}
  \leq\frac{2\delta|\xi|}{|\xi|^s}=2\delta|\xi|^{1-s}
  \leq2\delta\delta^{-\frac{1-s}{a}}=2\delta^{1+\frac{s-1}{a}}.$$
  \item[\rm(iii)] if $|\xi|\geq\delta^{-\frac{1}{a}}$, we get
  $$|m_{a,\mu}(\xi)|\lesssim\frac{1}{|\xi|^s}\leq\frac{1}{\delta^{-\frac{s}{a}}}\leq\delta^{1+\frac{s-1}{a}}.$$
\end{enumerate}
All of these imply that $\|m_{a,\mu}\|_{L^\infty(\mathbb{R}^n)}\lesssim\delta^{1+\frac{s-1}{a}}$ when $\beta>1$.

Next we consider the case of $\beta\leq1$,
\begin{enumerate}
  \item[\rm(i)] if $0\leq|\xi|\leq1$, we have
  $$|m_{a,\mu}(\xi)|=\left|\frac{e^{i(\delta^\beta\mu\cdot\xi +\delta|\xi|^a)}-1}{(1+|\xi|^2)^{\frac{s}{2}}}\right|\leq|\delta^\beta\mu\cdot\xi+\delta|\xi|^a|\leq\delta^\beta(|\xi|+|\xi|^a)\leq2\delta^\beta|\xi|^a\leq2\delta^\beta\leq2\delta^{\beta+\frac{s-1}{a}}.$$
  \item[\rm(ii)] if $1<|\xi|<\delta^{-\frac{1}{a}}$, we obtain
  $$|m_{a,\mu}(\xi)|\leq\frac{|\delta^\beta\mu\cdot\xi+\delta|\xi|^a|}{|\xi|^s}\leq\frac{\delta^\beta(|\xi|+|\xi|^a)}{|\xi|^s}
  \leq\frac{2\delta^\beta|\xi|}{|\xi|^s}=2\delta^\beta|\xi|^{1-s}
  \leq2\delta^\beta\delta^{-\frac{1-s}{a}}=2\delta^{\beta+\frac{s-1}{a}}.$$
  \item[\rm(iii)] if $|\xi|\geq\delta^{-\frac{1}{a}}$, we get
  $$|m_{a,\mu}(\xi)|\lesssim\frac{1}{|\xi|^s}\leq\frac{1}{\delta^{-\frac{s}{a}}}\leq\delta^{\beta+\frac{s-1}{a}}.$$
\end{enumerate}

This completes the proof of Theorem \ref{theorem 1.3}.
\end{proof}

\begin{proof}[Proof of Theorem \ref{theorem 1.4}]
Similarly to the proof of Theorem \ref{theorem 1.3}, it is enough to show that $\|m_{a,\mu}\|_{L^\infty(\mathbb{R}^n)}\lesssim\delta^{\frac{s}{a}}$ when $\beta>1$
and $\|m_{a,\mu}\|_{L^\infty(\mathbb{R}^n)}\lesssim\delta^{\beta-1+\frac{s}{a}}$ when $\beta\leq1$.
%with the bound does not depend on $\delta$, where $n\geq1$, $a\geq1$, $0<s\leq a$, $0<\delta<1$. In fact, when $\beta>1$,

When $\beta>1$,
\begin{enumerate}
  \item[\rm(i)] if $0\leq|\xi|\leq1$, we have
  $$|m_{a,\mu}(\xi)|=\left|\frac{e^{i(\delta^\beta\mu\cdot\xi +\delta|\xi|^a)}-1}{(1+|\xi|^2)^{\frac{s}{2}}}\right|\leq|\delta^\beta\mu\cdot\xi+\delta|\xi|^a|\leq\delta(|\xi|+|\xi|^a)
  \leq2\delta|\xi|\leq2\delta\leq2\delta^{\frac{s}{a}}.$$
  \item[\rm(ii)] if $1<|\xi|<\delta^{-\frac{1}{a}}$, we obtain
  $$|m_{a,\mu}(\xi)|\leq\frac{|\delta^\beta\mu\cdot\xi+\delta|\xi|^a|}{|\xi|^s}\leq\frac{\delta(|\xi|+|\xi|^a)}{|\xi|^s}
  \leq\frac{2\delta|\xi|^a}{|\xi|^s}=2\delta|\xi|^{a-s}
  \leq2\delta\delta^{-\frac{a-s}{a}}=2\delta^{\frac{s}{a}}.$$
  \item[\rm(iii)] if $|\xi|\geq\delta^{-\frac{1}{a}}$, we get
  $$|m_{a,\mu}(\xi)|\lesssim\frac{1}{|\xi|^s}\leq\frac{1}{\delta^{-\frac{s}{a}}}=\delta^{\frac{s}{a}}.$$
\end{enumerate}
When $\beta\leq1$,
\begin{enumerate}
  \item[\rm(i)] if $0\leq|\xi|\leq1$, we have
  $$|m_{a,\mu}(\xi)|=\left|\frac{e^{i(\delta^\beta\mu\cdot\xi +\delta|\xi|^a)}-1}{(1+|\xi|^2)^{\frac{s}{2}}}\right|\leq|\delta^\beta\mu\cdot\xi+\delta|\xi|^a|\leq\delta^\beta(|\xi|+|\xi|^a)
  \leq2\delta^\beta|\xi|\leq2\delta^\beta\leq2\delta^{\beta-1+\frac{s}{a}}.$$
  \item[\rm(ii)] if $1<|\xi|<\delta^{-\frac{1}{a}}$, we obtain
  $$|m_{a,\mu}(\xi)|\leq\frac{|\delta^\beta\mu\cdot\xi+\delta|\xi|^a|}{|\xi|^s}\leq\frac{\delta^\beta(|\xi|+|\xi|^a)}{|\xi|^s}
  \leq\frac{2\delta^\beta|\xi|^a}{|\xi|^s}=2\delta^\beta|\xi|^{a-s}
  \leq2\delta^\beta\delta^{-\frac{a-s}{a}}=2\delta^{\beta-1+\frac{s}{a}}.$$
  \item[\rm(iii)] if $|\xi|\geq\delta^{-\frac{1}{a}}$, we get
  $$|m_{a,\mu}(\xi)|\lesssim\frac{1}{|\xi|^s}\leq\frac{1}{\delta^{-\frac{s}{a}}}=\delta^{\frac{s}{a}}\leq\delta^{\beta-1+\frac{s}{a}}.$$
\end{enumerate}
This completes the proof of Theorem \ref{theorem 1.4}.
\end{proof}

In the proof of Theorem \ref{theorem 1.5}, it suffices to proof the Lemma \ref{lemma 3.3}.

\begin{lemma}\label{lemma 3.3}
Let $n\geq1$, $0<s\leq1$, $0<\delta<1$, $m_\gamma(\xi):=\frac{e^{i\delta\gamma(|\xi|)}-1}{(1+|\xi|^2)^{\frac{s}{2}}}$, $\gamma(t)$ be the same as in Theorem \ref{theorem 1.5}.
%$\gamma(t):\ \mathbb{R}^+\rightarrow \mathbb{R}^+$ and $\frac{\gamma(t)}{t}:\ \mathbb{R}^+\rightarrow \mathbb{R}^+$ are smooth and increasing on $(0,\infty)$ with $\gamma(0)\geq0$.
We have $\|m_\gamma\|_{L^\infty(\mathbb{R}^n)}\leq C\frac{1}{[\gamma^{-1}(\frac{\gamma(1)}{\delta})]^s}$, where the constant C does not depend on $\delta$.
\end{lemma}

\begin{proof}
The proof can also be divided into three situations.
\begin{enumerate}
  \item[\rm(i)] if $0\leq|\xi|\leq1$, from the fact that $0<\delta<1$ and $\frac{\gamma(t)}{t}:\ \mathbb{R}^+\rightarrow \mathbb{R}^+$ is smooth and increasing on $(0,\infty)$, we have that $\frac{\gamma(1)}{\delta}\leq \gamma(\frac{1}{\delta})$. Since $\gamma(t):\ \mathbb{R}^+\rightarrow \mathbb{R}^+$ is smooth and increasing on $(0,\infty)$, we further obtain that $\gamma^{-1}(\frac{\gamma(1)}{\delta})\leq \frac{1}{\delta}$. By $0<s\leq1$ and $0<\delta<1$, we assert that $[\gamma^{-1}(\frac{\gamma(1)}{\delta})]^s\leq [\frac{1}{\delta}]^s\leq\frac{1}{\delta}$. Consequently, we may conclude that $\delta\leq\frac{1}{[\gamma^{-1}(\frac{\gamma(1)}{\delta})]^s}$. The above implies
  $$|m_\gamma(\xi)|=\left|\frac{e^{i\delta\gamma(|\xi|)}-1}{(1+|\xi|^2)^{\frac{s}{2}}}\right|\leq\delta\gamma(|\xi|)\leq\delta\gamma(1)\lesssim\delta\leq\frac{1}{[\gamma^{-1}(\frac{\gamma(1)}{\delta})]^s}.$$
  \item[\rm(ii)] if $1<|\xi|<\gamma^{-1}(\frac{\gamma(1)}{\delta})$, from the fact that $0<s\leq1$ and $\frac{\gamma(t)}{t}:\ \mathbb{R}^+\rightarrow \mathbb{R}^+$ is smooth and increasing on $(0,\infty)$, which leads to $\frac{\gamma(t)}{t^s}=\frac{\gamma(t)}{t}t^{1-s}$ is also increasing on $(0,\infty)$ and further implies that
  $$|m_\gamma(\xi)|\leq\frac{\delta\gamma(|\xi|)}{|\xi|^s}\leq\delta\frac{\gamma[\gamma^{-1}(\frac{\gamma(1)}{\delta})]}{[\gamma^{-1}(\frac{\gamma(1)}{\delta})]^s}
  =\frac{\gamma(1)}{[\gamma^{-1}(\frac{\gamma(1)}{\delta})]^s}\lesssim \frac{1}{[\gamma^{-1}(\frac{\gamma(1)}{\delta})]^s}.$$
  \item[\rm(iii)] if $|\xi|\geq\gamma^{-1}(\frac{\gamma(1)}{\delta})$, from $t^s:\ \mathbb{R}^+\rightarrow \mathbb{R}^+$ is smooth and increasing on $(0,\infty)$, we have that
  $$|m_\gamma(\xi)|\lesssim\frac{1}{|\xi|^s}\leq\frac{1}{[\gamma^{-1}(\frac{\gamma(1)}{\delta})]^s}.$$
\end{enumerate}
This completes the proof of Lemma \ref{lemma 3.3}.
\end{proof}

In the proof of Theorem \ref{theorem 1.55}, it suffices to proof the Lemma \ref{lemma 3.33}.

\begin{lemma}\label{lemma 3.33}
Let $n\geq1$, $0<s\leq1$, $0<\delta<1$, $m_{\gamma,\mu}(\xi):=\frac{e^{i(\delta^\beta\mu\cdot\xi+\delta\gamma(|\xi|))}-1}{(1+|\xi|^2)^{\frac{s}{2}}}$, $\gamma(t)$ be the same as in Theorem \ref{theorem 1.5}.
 %$\gamma(t):\ \mathbb{R}^+\rightarrow \mathbb{R}^+$ and $\frac{\gamma(t)}{t}:\ \mathbb{R}^+\rightarrow \mathbb{R}^+$ are smooth and increasing on $(0,\infty)$ with $\gamma(0)\geq0$.
Then $\|m_{\gamma,\mu}\|_{L^\infty(\mathbb{R}^n)}\leq C\frac{1}{[\gamma^{-1}(\frac{\gamma(1)}{\delta})]^s}$ when $\beta>1$,
$\|m_{\gamma,\mu}\|_{L^\infty(\mathbb{R}^n)}\leq C\frac{\delta^{\beta-1}}{[\gamma^{-1}(\frac{\gamma(1)}{\delta})]^s}$ when $\beta\leq1$,
where the constant C does not depend on $\delta$.
\end{lemma}

\begin{proof}
Similarly to the proof of Lemma \ref{lemma 3.3}, the proof can also be divided into three situations.
When $\beta>1$,
\begin{enumerate}
  \item[\rm(i)] if $0\leq|\xi|\leq1$, we can also get $\gamma^{-1}(\frac{\gamma(1)}{\delta})\leq \frac{1}{\delta}$ and $\delta\leq\frac{1}{[\gamma^{-1}(\frac{\gamma(1)}{\delta})]^s}$.
  $$|m_{\gamma,\mu}(\xi)|=\left|\frac{e^{i(\delta^\beta\mu\cdot\xi+\delta\gamma(|\xi|))}-1}{(1+|\xi|^2)^{\frac{s}{2}}}\right|
  \leq\delta(|\xi|+\gamma(|\xi|))\leq\delta(1+\gamma(1))\lesssim\delta\leq\frac{1}{[\gamma^{-1}(\frac{\gamma(1)}{\delta})]^s}.$$
  \item[\rm(ii)] if $1<|\xi|<\gamma^{-1}(\frac{\gamma(1)}{\delta})$, from the fact that $\frac{\gamma(t)}{t^s}=\frac{\gamma(t)}{t}t^{1-s}$ is also increasing on $(0,\infty)$ with $0<s\leq1$ and further implies that
\begin{eqnarray*}
  |m_{\gamma,\mu}(\xi)|
  &\leq&\frac{\delta^\beta|\xi|+\delta\gamma(|\xi|)}{|\xi|^s}\\
  &\leq&\frac{\delta(|\xi|+\gamma(|\xi|))}{|\xi|^s}\\
  &=&\delta\left(|\xi|^{1-s}+\frac{\gamma(|\xi|)}{|\xi|^s}\right)\\
  &\leq&\delta\left([\gamma^{-1}(\frac{\gamma(1)}{\delta})]^{1-s}+\frac{\gamma[\gamma^{-1}(\frac{\gamma(1)}{\delta})]}{[\gamma^{-1}(\frac{\gamma(1)}{\delta})]^s}\right)\\
  &=&\delta[\gamma^{-1}(\frac{\gamma(1)}{\delta})]^{1-s}+\frac{\gamma(1)}{[\gamma^{-1}(\frac{\gamma(1)}{\delta})]^s}\\
  &\lesssim& \frac{1}{[\gamma^{-1}(\frac{\gamma(1)}{\delta})]^s}.
\end{eqnarray*}
  \item[\rm(iii)] if $|\xi|\geq\gamma^{-1}(\frac{\gamma(1)}{\delta})$, from $t^s:\ \mathbb{R}^+\rightarrow \mathbb{R}^+$ is smooth and increasing on $(0,\infty)$, we have that
  $$|m_{\gamma,\mu}(\xi)|\lesssim\frac{1}{|\xi|^s}\leq\frac{1}{[\gamma^{-1}(\frac{\gamma(1)}{\delta})]^s}.$$
\end{enumerate}
Next we consider the case of $\beta\leq1$.
\begin{enumerate}
  \item[\rm(i)] if $0\leq|\xi|\leq1$, we can also get $\gamma^{-1}(\frac{\gamma(1)}{\delta})\leq \frac{1}{\delta}$ and $\delta\leq\frac{1}{[\gamma^{-1}(\frac{\gamma(1)}{\delta})]^s}$.
  $$|m_{\gamma,\mu}(\xi)|=\left|\frac{e^{i(\delta^\beta\mu\cdot\xi+\delta\gamma(|\xi|))}-1}{(1+|\xi|^2)^{\frac{s}{2}}}\right|
  \leq\delta^\beta(|\xi|+\gamma(|\xi|))\leq\delta^\beta(1+\gamma(1))\lesssim\delta^\beta
  \lesssim\frac{\delta^{\beta-1}}{[\gamma^{-1}(\frac{\gamma(1)}{\delta})]^s}.$$
  \item[\rm(ii)] if $1<|\xi|<\gamma^{-1}(\frac{\gamma(1)}{\delta})$, from the fact that $\frac{\gamma(t)}{t^s}=\frac{\gamma(t)}{t}t^{1-s}$ is also increasing on $(0,\infty)$ with $0<s\leq1$ and further implies that
 \begin{eqnarray*}
 |m_{\gamma,\mu}(\xi)|
 &\leq&\frac{\delta^\beta|\xi|+\delta\gamma(|\xi|)}{|\xi|^s}\\
  &\leq&\frac{\delta^\beta(|\xi|+\gamma(|\xi|))}{|\xi|^s}\\
  &=&\delta^\beta\left(|\xi|^{1-s}+\frac{\gamma(|\xi|)}{|\xi|^s}\right)\\
  &\leq&\delta^\beta\left([\gamma^{-1}(\frac{\gamma(1)}{\delta})]^{1-s}+\frac{\gamma[\gamma^{-1}(\frac{\gamma(1)}{\delta})]}{[\gamma^{-1}(\frac{\gamma(1)}{\delta})]^s}\right)\\
  &\lesssim&\delta^\beta[\gamma^{-1}(\frac{\gamma(1)}{\delta})]^{1-s}+\frac{\delta^{\beta-1}}{[\gamma^{-1}(\frac{\gamma(1)}{\delta})]^s}\\
  &\lesssim &\frac{\delta^{\beta-1}}{[\gamma^{-1}(\frac{\gamma(1)}{\delta})]^s}.
  \end{eqnarray*}
  \item[\rm(iii)] if $|\xi|\geq\gamma^{-1}(\frac{\gamma(1)}{\delta})$, from $t^s:\ \mathbb{R}^+\rightarrow \mathbb{R}^+$ is smooth and increasing on $(0,\infty)$, we have that
  $$|m_{\gamma,\mu}(\xi)|\lesssim\frac{1}{|\xi|^s}\leq\frac{1}{[\gamma^{-1}(\frac{\gamma(1)}{\delta})]^s}
  \lesssim\frac{\delta^{\beta-1}}{[\gamma^{-1}(\frac{\gamma(1)}{\delta})]^s}.$$
\end{enumerate}

This completes the proof of Lemma \ref{lemma 3.33}.
\end{proof}

\section*{Acknowledgements}

The authors would like to thank Prof. Junfeng Li for many valuable comments and useful discussions.

\bigskip

\noindent  Dan Li

\smallskip

\noindent  Laboratory of Mathematics and Complex Systems
(Ministry of Education of China),
School of Mathematical Sciences, Beijing Normal University,
Beijing 100875, People's Republic of China

\smallskip

\noindent {\it E-mails}: \texttt{danli@mail.bnu.edu.cn}

\bigskip

\noindent Haixia Yu (Corresponding author)

\smallskip

\noindent Department of Mathematics, Sun Yat-sen University, Guangzhou, 510275,  People's Republic of China

\smallskip

\noindent{\it E-mail}: \texttt{yuhaixia@mail.bnu.edu.cn}

\end{document}